\documentclass[12pt,oneside]{amsart}

\usepackage{amsthm,amsmath,amssymb}
\usepackage{mathrsfs}
\usepackage{enumerate}
\usepackage{amsfonts}
\usepackage{verbatim}
\usepackage{amsbsy}
\usepackage{amsmath}
\usepackage{amssymb}
\usepackage{MnSymbol}
\usepackage{mathrsfs}
\usepackage{xcolor}
\usepackage{cite}

\newtheorem{theorem}{Theorem}[section]
\newtheorem{lemma}[theorem]{Lemma}
\newtheorem{proposition}[theorem]{Proposition}

\newtheorem*{claim}{Claim}
\newtheorem{question}[theorem]{Question}
\newtheorem{corollary}[theorem]{Corollary}
\theoremstyle{definition}
\newtheorem{definition}[theorem]{Definition}

\theoremstyle{remark}
\newtheorem{remark}[theorem]{Remark}
\newtheorem{example}[theorem]{Example}

\DeclareMathOperator{\Adm}{Adm}
\DeclareMathOperator{\OT}{OT}

\DeclareMathOperator{\FR}{FR}

\begin{document}

\title{Ramsey Orderly Algebras as a New Approach to Ramsey Algebras}
\author{Wen Chean Teh}
\address{School of Mathematical Sciences\\
Universiti Sains Malaysia\\
11800 USM, Malaysia}
\email[Corresponding author]{dasmenteh@usm.my}
\author{Zu Yao Teoh}
\address{School of Mathematical Sciences\\
Universiti Sains Malaysia\\
11800 USM, Malaysia}
\email{teohzuyao@gmail.com}

\begin{abstract}
Ramsey algebras are algebras that induce Ramsey spaces, which are generalizations of the Ellentuck space and the Milliken space. Previous work has suggested a possible local version of Ramsey algebras induced by infinite
sequences. We formulate this local version and call it Ramsey orderly algebra. In this paper, we present an introductory treatment of this new notion and provide justification for it to be a sound approach for further study in Ramsey algebras. The main connection is that an algebra is Ramsey if and only if each of its induced orderly algebra is Ramsey.
\end{abstract}


\maketitle

\section{Introduction}

Ramsey spaces as introduced by Carlson in \cite{tC88} are generalizations of the Ellentuck space \cite{eE74} and the Milliken space \cite{kM75}.
This notion of Carlson has since then attracted a considerable amount of interest due to its power to derive a proliferation of known standard Ramsey
theoretic results based on the existence of certain Ramsey spaces of variable words \cite{tC88}. Examples of standard results are the
dual Ellentuck Theorem \cite{CS84}, the Graham-Rothschild theorem on \mbox{$n$-parameter} sets \cite{GR71}, and the Hales-Jewett Theorem \cite{HJ63}.
A modern reference to the topic is \cite{sT10}.

The topological property of a Ramsey space can, in fact, be reduced to some combinatorial property. This is
essentially captured by the abstract version of Ellentuck's Theorem, first pointed out by Carlson in the same article. In the context of a Ramsey space
associated to an algebra, such as a Ramsey space of variable words, that combinatorial property
is reflected in the algebra itself. For this reason, Carlson conceived the notion of Ramsey algebras and its first results followed suit through the
work of the first  author \cite{wcT13b,wcT13a,wcT13,wcT13c}.

Every semigroup is a Ramsey algebra. The author's work in~\cite{wcT13a}
generalizes the result to any algebra having sufficiently many infinite sequences of the underlying set that induce
what are here called orderly semigroups (see Example~\ref{1211b}). It is this observation that motivated the introduction of orderly algebras
(Definition~\ref{1511a}). This notion shifts the subject of Ramsey algebra from a global perspective to a local, sequential one.
Specifically, the characterization of Ramsey algebras can be reduced to the characterization of Ramsey orderly algebras.

This paper presents preliminary findings to justify the pertinence of this notion for the study of Ramsey algebras. The authors believe that
approaching the subject of Ramsey algebras through orderly algebras might be better suited for further studies of the subject. Besides that, the paper
aims to disseminate the subject of Ramsey algebras to the combinatorial and logical community to encourage participation in the study of this
subject, which is at its infancy stage. One open problem in the subject is the characterization of Ramsey algebras in terms of the property of the
underlying operations. At this point of time, little is known about which common algebras, apart from semigroups, are Ramsey.


\section{Preliminary}
The set of natural numbers and the set of positive integers are denoted by $\omega$ and $\mathbb{N}$ respectively.
The set of infinite sequences in $A$ is denoted by $^{\omega}\!A$. The cardinality of  a set $A$ is denoted by $\vert A\vert$.


To us an \emph{algebra} is a pair $(A, \mathcal{F})$, where $A$ is a nonempty set and $\mathcal{F}$ is a collection of operations on $A$, none of which
is nullary. We will write $(A,\{f\})$ simply as $(A,f)$.

We assume the reader is familiar with the syntax and semantics of first order logic. In particular, we are only concerned with purely functional
first-order logic $\mathcal{L}$, that is, logic whose non-logical symbols are functional. 
Fix a list of (syntactic) variables, $v_0,v_1,v_2, \dotsc$.  The \emph{index} of $v_i$ is $i$.
The \emph{terms} of such a language are expressions built up in the usual way from the variables and the function symbols.
An \emph{$\mathcal{L}$-algebra} is an $\mathcal{L}$-structure in the usual sense.


Suppose that $\mathfrak{A}$ is an $\mathcal{L}$-algebra. An \emph{assignment} is (identified with) an infinite sequence whose terms are elements of the
\emph{universe} $\Vert \mathfrak{A}\Vert$ of $\mathfrak{A}$.  The \emph{interpretation} of a term $t$ under $\mathfrak{A}$ and assignment $\vec{a}$, denoted
$t^{\mathfrak{A}}[\vec{a}]$, is defined inductively in the usual way.


For each algebra $\mathfrak{A}=(A,\mathcal{F})$, there is a natural language associated to it. Let $\mathcal{L}_{\mathcal{F}}$ denote the language $\{
\,\underline{f} \mid f \in \mathcal{F}\,\}$ where $\underline{f}$ and $f$ have the same arity for every $f \in \mathcal{F}$ and $\underline{f}$ and
$\underline{g}$ are distinct whenever $f$ and $g$ are distinct. We will identify  $\mathfrak{A}$  with the $\mathcal{L}_{\mathcal{F}}$-algebra whose
universe is $A$ and the interpretation of $\underline{f}$ is $f$ for every $f\in \mathcal{F}$.

To discuss Ramsey algebras or Ramsey orderly algebras, the notion of a ``term" is inadequate. We need a stronger notion.

\begin{definition}\label{1112d}
Suppose that $\mathcal{L}$ is a language. An \emph{orderly term} of $\mathcal{L}$ is a term of $\mathcal{L}$ such that the indices  of the variables appearing in
it from left to right is strictly increasing. The set of orderly terms of $\mathcal{L}$ is denoted by $\OT(\mathcal{L})$.
If $t,t'\in \OT(\mathcal{L})$, then $t < t'$ means that the index of the last variable in $t$ is less than the index of the first variable in $t'$.
An infinite  sequence $\vec{t}$ of orderly terms of $\mathcal{L}$ is \emph{admissible} if{f}  it is increasing with respect to $<$. The set of
admissible sequences of $\mathcal{L}$ is denoted by
$\Adm(\mathcal{L})$.
\end{definition}

\begin{definition}\label{0524c}
Suppose that $\mathfrak{A}=(A, \mathcal{F})$ is an algebra and $\vec{a}, \vec{b}\in {^\omega}\!A$.
We say that $\vec{b}$ is a \emph{reduction} of $\vec{a}$ (with respect to $\mathcal{F}$), denoted $\vec{b} \leq_{\mathcal{F}} \vec{a}$, if{f} there
exists  $\vec{t}\in\Adm(\mathcal{L}_{\mathcal{F}})$ such that $\vec{b}(i)=(\vec{t}(i))^{\mathfrak{A}}[\vec{a}]$ for all $i \in \omega$.
\end{definition}

The relation $\leq_{\mathcal{F}}$ is reflexive and transitive.

\begin{definition}
Suppose that $\mathfrak{A}=(A, \mathcal{F})$ is an $\mathcal{L}$-algebra and $\vec{a}\in {^\omega}\!A$. The \emph{set  of finite reductions} of
$\vec{a}$ (with respect to $\mathcal{F}$),
denoted  $\FR_{\mathcal{F}}(\vec{a})$,  is defined by
\[\FR_{\mathcal{F}}(\vec{a})=\{\,t^{\mathfrak{A}}[\vec{a}]\mid t \in \OT(\mathcal{L})\,  \}.\]
\end{definition}

\begin{definition}
Suppose $(A, \mathcal{F})$ is an algebra.
We say that $(A, \mathcal{F})$ is \emph{Ramsey} if{f}
for every $\vec{a}\in {^\omega}\!A$ and $X \subseteq A$, there exists $\vec{b} \leq_{\mathcal{F}} \vec{a}$ such that $\FR_{\mathcal{F}}(\vec{b})$ is
either contained in or disjoint from $X$.
\end{definition}

If for every $\vec{a} \in {^\omega}\!A$, there exists  $\vec{b} \leq_{\mathcal{F}} \vec{a}$ such that  $\vert \FR_{\mathcal{F}}(\vec{b}) \vert=1$, then
$(A, \mathcal{F})$ is trivially Ramsey, and we say that it is a \emph{degenerate Ramsey algebra}.

The following localized version of Ramsey algebra is very relevant to this work.

\begin{definition}
Suppose $(A,\mathcal{F})$ is an algebra and $\vec{a}\in {^\omega}\!A$. We say that $(A, \mathcal{F})$ is \emph{Ramsey below $\vec{a}$} if{f}
for every $\vec{b} \leq_{\mathcal{F}}\vec{a}$ and  $X \subseteq A$, there exists $\vec{c} \leq_{\mathcal{F}} \vec{b}$ such that
$\FR_{\mathcal{F}}(\vec{c})$ is either contained in or disjoint from $X$.
\end{definition}

\begin{remark}\label{0102a}
An algebra $(A, \mathcal{F})$ is Ramsey if and only if it is Ramsey below $\vec{a}$ for every $\vec{a}\in {^\omega}\!A$. Additionally, if
$(A,\mathcal{F})$ is Ramsey below $\vec{a}$, then it is Ramsey below $\vec{b}$ whenever $\vec{b}\leq_{\mathcal{F}}\vec{a}$.
\end{remark}

We now address the intimate connection between Ramsey algebras and Ramsey spaces. The definition of a Ramsey space presented in our context here mirrors
that given in \cite{tC88}. A Ramsey space is called a topological Ramsey space in \cite{sT10} and it has a slightly different axiomatization.

\begin{definition}
A \emph{preorder with approximations} is a pair
$\mathfrak{R}=(R, \leq)$ such that $R$ is a nonempty set of infinite sequences and $\leq$ is a reflexive and transitive relation on $R$. For $n \in
\omega$ and $\vec{a}\in R$, define
$$[n,\vec{a}]=\{\,\vec{b}\in R\mid \vec{b}\leq \vec{a} \text{ and }
\langle \vec{b}(0), \dotsc, \vec{b}(n-1)\rangle= \langle \vec{a}(0), \dotsc, \vec{a}(n-1)\rangle \,\}.$$
The \emph{natural topology} on $\mathfrak{R}$ is the topology generated by the sets $[n,\vec{a}]$.
\end{definition}

\begin{definition}
Suppose $\mathfrak{R}= (R, \leq)$ is a preorder with approximations. Assume $X$ is a subset of $R$. We say that $X$ is \emph{Ramsey} if{f} for every $n
\in \omega$ and $\vec{a} \in R$, there exists $\vec{b} \in [n,\vec{a}]$ such that $[n,\vec{b}]$ is either contained in or disjoint from $X$. Assuming
the Axiom of Choice, $\mathfrak{R}$ is a \emph{Ramsey space} if{f}  every subset of $\mathfrak{R}$ (endowed with the natural topology) having the
property of Baire is Ramsey.
\end{definition}

Our work concerns preorders with approximations associated to algebras.

\begin{definition}
Suppose $(A,\mathcal{F})$ is an algebra  and $\vec{a}\in {^\omega}\!A$. Define $\mathscr{R}(A,\mathcal{F})$ to be the preorder with approximations
$({^\omega}\!A, \leq_{\mathcal{F}})$ and $\mathscr{R}_{\vec{a}}(A,\mathcal{F})$
to be the preorder with approximations $(\{\, \vec{b}\in {^\omega}\!A\mid \vec{b} \leq_{\mathcal{F}}\vec{a}\,\},\leq_{\mathcal{F}})$.
\end{definition}

\begin{theorem}\label{1129b}
Suppose $\mathcal{F}$ is a finite collection of operations on a set $A$, none of which is unary, and $\vec{a}\in {^\omega}\!A$. Then
\begin{enumerate}
\item $\mathscr{R}(A, \mathcal{F})$ is a Ramsey space if and only if $(A, \mathcal{F})$  is a Ramsey algebra;
\item $\mathscr{R}_{\vec{a}}(A, \mathcal{F})$ is a Ramsey space if and only if $(A, \mathcal{F})$  is Ramsey below $\vec{a}$.
\end{enumerate}
\end{theorem}

Theorem~\ref{1129b}\textit{(1)} is a version of Lemma~4.14 in \cite{tC88} using the notion of Ramsey algebra. In fact, $\mathscr{R}(A, \mathcal{F})$ was
shown in \cite{tC88} to satisfy the assumptions for the abstract version of Ellentuck's Theorem that play a key role. As Theorem~\ref{1129b}\textit{(2)} can be proved analogously, it is first stated in \cite{wcT13} without proof.
Furthermore, Theorem~\ref{1129b}\textit{(1)} is strengthen in \cite{wcT13} to allow for the collection of operations to be appended by any collection of
unary operations. As of this writing, the author claims that the same conclusion holds provided the underlying set of the algebra is countable. However, it remains unclear
whether there exists a Ramsey algebra such that the corresponding space is not Ramsey.

\begin{example}
The empty algebra $(\omega,\emptyset)$ is Ramsey precisely due to the pigeonhole principle. Hence, $\mathscr{R}(\omega, \emptyset)$ is a Ramsey space.
Identifying infinite subsets $A$ and $B$ of $\omega$ with strictly increasing sequences of natural numbers, we have $A\subseteq B$ if and only if
$A\leq_\emptyset B$. Therefore, the Ellentuck space is actually isomorphic to
the subspace of $\mathscr{R}(\omega, \emptyset)$ induced by the set $\{\, A\in {^\omega}\omega\mid A \text{ is strictly inceasing}   \,\}$.
This set is the basic open set $[0,\langle 0,1,2,\dotsc\rangle]$ in the natural topology on $\mathscr{R}(\omega, \emptyset)$ and as such induces a
Ramsey subspace.
\end{example}

\begin{example}
Suppose $L$ is a finite alphabet and $v$ is  a distinct variable not contained in $L$.
A \emph{variable word over $L$} is a finite sequence $w$ of elements of $L \cup \{v\}$ such that the variable $v$ occurs at least once in $w$.
Denote the set of variable words over $L$ by $W$. If $w\in W$ and $a\in L$, then $w(a)$ is the result of replacing every occurrence of $v$ in $w$ by
$a$. The concatenation of two variable words $w$ and $w'$ is denoted by $w \ast w'$.
Let $\mathcal{F}$ consist of $\ast$ and the following binary operations on $W$:
$$(w,w')\mapsto w\ast w'(a) \quad , \quad a\in L;$$
$$(w,w')\mapsto w(a)\ast w' \quad , \quad a\in L.$$
$\mathscr{R}(W,\mathcal{F})$ is a prototype of Ramsey spaces of variable words \cite{tC88}.
 \end{example}


Hindman's Theorem \cite{nH74} implies that $(\mathbb{N}, +)$ is a Ramsey algebra. In fact, the generalization in the next theorem is essentially due to
Hindman's Theorem  (see \cite[Section 5.2]{HS12}). Alternatively,
it follows from the fact that the preorder with approximations associated to a semigroup is a Ramsey space (Theorem~6 in \cite{tC88}).

\begin{theorem}\label{050613a}
Every semigroup is a Ramsey algebra.
\end{theorem}

\begin{remark}\label{Moufang}
Structures that satisfy the Moufang identities are very close to being semigroups.
For example, under octonion multiplication, the nonzero octonions form a Moufang loop that is nonassociative.
Nevertheless, this Moufang loop is not a Ramsey algebra (see \cite{rdz16}).
\end{remark}

If $\mathscr{R}(A,\mathcal{F})$ is a Ramsey space, then $(A,\mathcal{F})$ is trivially a Ramsey algebra because every subset of $A$ induces a subset of
$^{\omega}\!A$ that is clopen in the natural topology. In fact, the strength of being a Ramsey space bestows more Ramsey type property on the associated
algebra, which is the reason behind the following analogue of the Milliken-Taylor Theorem \cite{kM75,aT76}.

\begin{theorem}\label{0910a}
Suppose $\mathcal{F}$ is a finite collection of operations on a set $A$, none of which is unary, $\vec{a}\in {^\omega}\!A$, and $n\in \omega$. If
$\mathfrak{A}=(A, \mathcal{F})$  is Ramsey below $\vec{a}$, then for every $\vec{b}\leq_{\mathcal{F}}\vec{a}$ and $X\subseteq A^n$, there exists
$\vec{c}\leq_{\mathcal{F}}\vec{b}$ such that
$[\FR_{\mathcal{F}}(\vec{c})]^n_<$ is either contained in or disjoint from $X$, where
$$[\FR_{\mathcal{F}}(\vec{c})   ]^n_<= \{\,(t_1^{\mathfrak{A}}[\vec{c}],\dotsc,t_n^{\mathfrak{A}}[\vec{c}])  \mid t_1,\dotsc,t_n \in
\OT(\mathcal{L}_{\mathcal{F}})\text{ with } t_1<\dotsb<t_n   \,  \}.$$
\end{theorem}

\begin{proof}
Let $R=\{\, \vec{b}\in {^\omega}\!A\mid \vec{b} \leq_{\mathcal{F}}\vec{a}\,\}$.   By Theorem~\ref{1129b}, $\mathscr{R}_{\vec{a}}(A,
\mathcal{F})=(R,\leq_{\mathcal{F}})$ is a Ramsey space.

Fix $\vec{b}\leq_{\mathcal{F}}\vec{a}$ and $X\subseteq A^n$. Let
$$Y=\{\, \vec{d}\in R\mid (\vec{d}(0), \vec{d}(1), \dotsc, \vec{d}(n-1)) \in X                \,\}.$$
Clearly, $[n,\vec{d}]\subseteq Y$ for every $\vec{d}\in Y$. Hence,
$Y=\bigcup_{\vec{d}\in Y} [n,\vec{d}]$ is open in the natural topology on $\mathscr{R}_{\vec{a}}(A, \mathcal{F})$.
Thus $Y$ has the property of Baire and so is Ramsey because $\mathscr{R}_{\vec{a}}(A, \mathcal{F})$ is Ramsey.
Therefore,  we can choose $\vec{c}\leq_{\mathcal{F}}\vec{b}$ such that $[0,\vec{c}]=
\{\,\vec{d}\in R\mid \vec{d}\leq_{\mathcal{F}}\vec{c}   \,\}$ is either contained in or disjoint from $Y$.

Suppose $(d_0, d_1, \dotsc, d_{n-1})\in [\FR_{\mathcal{F}}(\vec{c})]^n_<$ is arbitrary. By the definition of a reduction, $\langle d_0,d_1, \dotsc,
d_{n-1} \rangle$ can be extended easily to some reduction $\vec{d}$ of $\vec{c}$, meaning $\vec{d}\in [0,\vec{c}]$.
Note that $\vec{d}\in Y$ if and only if $(d_0, d_1, \dotsc, d_{n-1})\in X$. It follows that
$[\FR_{\mathcal{F}}(\vec{c})]^n_<$ is either contained in or disjoint from $X$.
\end{proof}

\section{Orderly Algebras and Their Reductions}
We now introduce the notion of orderly algebra. Its naturality and relevance to the study of Ramsey algebras will become clear by the time we reach
Theorem~\ref{1908a}.

\begin{definition}\label{1511a}
Suppose that $\mathcal{L}$ is a language. An \emph{orderly $\mathcal{L}$-algebra} is a function $\mathscr{A}$ with domain $\OT(\mathcal{L})$ such that
for each  $f\in \mathcal{L}$ the following holds: if $f$ is $n$-ary, then $\mathscr{A}(ft_1t_2\dotsm t_n)=\mathscr{A}(ft_1't_2'\dotsm t_n')$ whenever
$t_1,t_2,\dotsc, t_n,t_1',t_2', \dotsc,t_n'$ are orderly terms of $\mathcal{L}$ with $t_1<t_2<\dotsb<t_n$ and $t_1'<t_2'<\dotsb<t_n'$ such that
$\mathscr{A}(t_k)=\mathscr{A}(t_k')$ for $1\leq k\leq n$.
The range of $\mathscr{A}$, denoted $\Vert \mathscr{A}\Vert$, is called the \emph{universe} of $\mathscr{A}$.
\end{definition}

\begin{example}\label{constantexample}
Suppose that $\mathcal{L}$ is a language. If $\mathscr{A}$ is any constant function with domain $\OT(\mathcal{L})$, then $\mathscr{A}$ is what we call a
\emph{trivial orderly $\mathcal{L}$-algebra}.
\end{example}

\begin{example}\label{indexexample}
Suppose that $\mathcal{L}$ is a language. Let $\mathscr{A}(t)= \{ i \in \omega\mid v_i \text{ appears in } t \}$
for all $t\in\OT(\mathcal{L})$. Then $\mathscr{A}$ is an orderly $\mathcal{L}$-algebra.
\end{example}

Examples \ref{constantexample} and \ref{indexexample} are arbitrary in the sense that no particular algebra has played a role in defining the values of
$\mathscr{A}$ in either case. Since we are interested in algebras, we would like to be able to obtain an orderly $\mathcal{L}$-algebra from a given
algebra. This is what we do in the next definition.

\begin{definition}
Suppose $\mathcal{L}$ is a language, $\mathfrak{A}$ is an $\mathcal{L}$-algebra and $\vec{a}\in {^\omega}\Vert \mathfrak{A}\Vert$.
The \emph{orderly $\mathcal{L}$-algebra induced from $\mathfrak{A}$ by $\vec{a}$}, denoted $\mathfrak{A}_{\vec{a}}$, is defined by
$$\mathfrak{A}_{\vec{a}}(t)=t^{\mathfrak{A}}[\vec{a}]\;\text{for all}\;t\in \OT(\mathcal{L}).$$
\end{definition}

The fact that $\mathfrak{A}_{\vec{a}}$ is well-defined follows from the semantics of first order logic, namely
$(ft_1t_2\dotsm t_n)^{\mathfrak{A}}[\vec{a}]=f^{\mathfrak{A}}(t_1^{\mathfrak{A}}[\vec{a}], \dotsc, t_n^{\mathfrak{A}}[\vec{a}])$
whenever $f\in \mathcal{L}$, say $n$-ary, and $t_1, \dotsc, t_n$ are terms of $\mathcal{L}$.

Note that if $\mathfrak{A}=(A,\mathcal{F})$ is an $\mathcal{L}$-algebra and $\vec{a}\in {^\omega}\!A$, then $\Vert \mathfrak{A}_{\vec{a}}\Vert=
\FR_{\mathcal{F}}(\vec{a})$. 

\begin{example}\label{1211b}
Suppose $\mathcal{L}=\{f\}$, where $f$ is binary. An orderly $\mathcal{L}$-algebra $\mathscr{A}$  is an \emph{orderly semigroup}
if{f} $ \mathscr{A}(ff t_1t_2t_3)=\mathscr{A}(ft_1f  t_2t_3)$
for every $t_1,t_2,t_3 \in \OT( \mathcal{L}  )$ with $t_1 < t_2 < t_3$.
If $\mathscr{A}$ is an orderly semigroup, then in fact $\mathscr{A}(s)=\mathscr{A}(t)$ whenever the same variables occur in $s$ and $t$ (essentially
Lemma~4.2 in \cite{wcT13a}). Every orderly $\mathcal{L}$-algebra induced from a semigroup is an orderly semigroup.
\end{example}

It will be shown now that, conversely, every orderly $\mathcal{L}$-algebra is induced from some $\mathcal{L}$-algebra.

Suppose that $\mathcal{L}$ is a language and $\mathscr{A}$ is an orderly $\mathcal{L}$-algebra.
For each fixed $f\in \mathcal{L}$, say $f$ is $n$-ary, the map
$$(\mathscr{A}(t_1), \dotsc, \mathscr{A}(t_n)) \mapsto \mathscr{A}( ft_1\dotsm t_n) \quad,\quad t_1<t_2<\dotsb<t_n $$
is a well-defined $n$-ary partial operation
on $\Vert \mathscr{A}\Vert$. Extend this map arbitrarily to an $n$-ary operation on $\Vert \mathscr{A}\Vert$, denoted by $f^*$.

Let $\mathfrak{A}=(\Vert \mathscr{A}\Vert, \{ f^*\}_{f\in \mathcal{L}})$ and $\vec{a}=\langle \mathscr{A}(v_0), \mathscr{A}(v_1),
\mathscr{A}(v_2), \dotsb\rangle$.
Then $\mathfrak{A}$ is canonically an $\mathcal{L}$-algebra with $f^{\mathfrak{A}}=f^*$ for each $f\in \mathcal{L}$.
We claim that $\mathfrak{A}_{\vec{a}}=\mathscr{A}$.
To see this, we argue by induction on the complexity of orderly terms that $\mathfrak{A}_{\vec{a}}(t)=\mathscr{A}(t)$ for all $t\in \OT(\mathcal{L})$.
For each variable $v_i$, we have $\mathfrak{A}_{\vec{a}}(v_i)=v_i^{\mathfrak{A}}[\vec{a}]=\vec{a}(i)=\mathscr{A}(v_i)$.
Now, assume $t=ft_1t_2\dotsm t_n$ for some $n$-ary $f\in \mathcal{L}$ and orderly terms $t_1<t_2<\dotsb< t_n$.
By the induction hypothesis, $\mathfrak{A}_{\vec{a}}(t_i)= t_i^{\mathfrak{A}}[\vec{a}]= \mathscr{A}(t_i)$ for each $1\leq i\leq n$.
By definition,  $\mathfrak{A}_{\vec{a}}(t)=t^{\mathfrak{A}}[\vec{a}]= f^{\mathfrak{A}}(t_1^{\mathfrak{A}}[\vec{a}], t_2^{\mathfrak{A}}[\vec{a}], \dotsc,
t_n^{\mathfrak{A}}[\vec{a}])=  f^\ast(\mathscr{A}(t_1), \mathscr{A}(t_2),   \dotsc, \mathscr{A}(t_n))
=\mathscr{A}(ft_1t_2\dotsm t_n)=\mathscr{A}(t)$. Hence, we have proved the following theorem.

\begin{theorem}\label{0310a}
Suppose that $\mathcal{L}$ is a language and $\mathscr{A}$ is an orderly $\mathcal{L}$-algebra. Then there exists an $\mathcal{L}$-algebra
$\mathfrak{A}$ with universe $\Vert \mathscr{A}\Vert$ such that $\mathscr{A}$ is induced from $\mathfrak{A}$.
\end{theorem}



\begin{definition}
Suppose $\mathcal{L}$ is a language, $s\in \OT(\mathcal{L})$, and $\vec{t}  \in \Adm(\mathcal{L})  $.
Define $s[\vec{t}]$ to be the orderly term of $\mathcal{L}$ obtained by replacing each variable $v_i$ occurring in $s$ by $\vec{t}(i)$.
\end{definition}

\begin{definition}\label{2708a}
Suppose $\mathcal{L}$ is a language and suppose $\mathscr{A}$ and  $\mathscr{B}$ are orderly $\mathcal{L}$-algebras.
We say that $\mathscr{B}$ is a \emph{reduction} of $\mathscr{A}$ if{f} there exists $\vec{t}  \in \Adm(\mathcal{L})  $ such that
for every  $s\in \OT(\mathcal{L})$,
$$\mathscr{B}(s)=\mathscr{A}(s[\vec{t}]) .$$
We say that $\mathscr{B}$ is a reduction of $\mathscr{A}$ \emph{witnessed} by $\vec{t}$. 
\end{definition}

\begin{remark}\label{2009a}
\begin{enumerate}
\item  If $\mathscr{B}$ is a reduction of $\mathscr{A}$, then $\Vert \mathscr{B}\Vert \subseteq \Vert \mathscr{A}\Vert$.
\item  If $\mathscr{B}$ is a reduction of $\mathscr{A}$ and $\mathscr{C}$ is a reduction of $\mathscr{B}$, then $\mathscr{C}$ is a reduction of
    $\mathscr{A}$.
\end{enumerate}
\end{remark}

The choice of the term ``reduction" is justified by the following proposition.

\begin{proposition}\label{2008a}
Suppose $\mathcal{L}$ is a language, $\mathfrak{A}=(A,\mathcal{F})$ is an $\mathcal{L}$-algebra, and $\vec{a}\in {^\omega}\!A$. Then
\begin{enumerate}
\item $\vec{b} \leq_{\mathcal{F}}\vec{a}$ if and only if $\mathfrak{A}_{\vec{b}}$ is a reduction of $\mathfrak{A}_{\vec{a}}$ for each $\vec{b}\in
    {^\omega}\!A$;
\item if $\mathscr{B}$ is a reduction of $\mathfrak{A}_{\vec{a}}$ and
 $\vec{b}(i)=\mathscr{B}(v_i)$ for all $i\in \omega$, then $\vec{b}$ is a reduction of $\vec{a}$ and $\mathscr{B}=\mathfrak{A}_{\vec{b}}$.
\end{enumerate}
\end{proposition}

\begin{proof}
For the first part, fix $\vec{b}\in {^\omega}\!A$. Assume  $\vec{b} \leq_{\mathcal{F}}\vec{a}$. Then choose $\vec{t}\in \Adm(\mathcal{L}_{\mathcal{F}})$
such that
$\vec{b}(i)=(\vec{t}(i))^{\mathfrak{A}}[\vec{a}]$ for all $i\in \omega$.
For every $s\in \OT(\mathcal{L}_{\mathcal{F}})$, by the substitution lemma,
$\mathfrak{A}_{\vec{b}}(s)=s^{\mathfrak{A}}[\vec{b}]=(s[\vec{t}])^{\mathfrak{A}}[\vec{a}]=\mathfrak{A}_{\vec{a}}(s[\vec{t}] )$.
Thus $\mathfrak{A}_{\vec{b}}$ is a reduction of
$\mathfrak{A}_{\vec{a}}$ witnessed by $\vec{t}$.
Conversely, assume $\mathfrak{A}_{\vec{b}}$ is a reduction of $\mathfrak{A}_{\vec{a}}$ witnessed by some $\vec{t}\in \Adm(\mathcal{L}_{\mathcal{F}})$.
Then
$\vec{b}(i)=v_i^{\mathfrak{A}}[\vec{b}]=\mathfrak{A}_{\vec{b}}(v_i)=\mathfrak{A}_{\vec{a}}(\vec{t}(i)) =  (\vec{t}(i))^{\mathfrak{A}}[\vec{a}]$ for all
$i\in \omega$. Therefore, $\vec{b}$ is a reduction of $\vec{a}$.

For the second part, suppose $\mathscr{B}$ is a reduction of $\mathfrak{A}_{\vec{a}}$ witnessed by some $\vec{t}\in \Adm(\mathcal{L}_{\mathcal{F}})$ and
$\vec{b}(i)=\mathscr{B}(v_i)$ for all $i\in \omega$.
Then $\vec{b}(i)= \mathfrak{A}_{\vec{a}}(\vec{t}(i))= ( \vec{t}(i))^{\mathfrak{A}}[\vec{a}]$ for all $i\in \omega$, implying that $\vec{b}$ is a
reduction of $\vec{a}$. Now, for every $s\in \OT(\mathcal{L}_{\mathcal{F}})$,
$\mathscr{B}(s)=\mathfrak{A}_{\vec{a}}(s[\vec{t}] )=(s[\vec{t}])^{\mathfrak{A}}[\vec{a}]
=s^{\mathfrak{A}}[\vec{b}]=\mathfrak{A}_{\vec{b}}(s)$. Therefore, $\mathscr{B}= \mathfrak{A}_{\vec{b}}$.
\end{proof}


\begin{corollary}\label{0310c}
Suppose $\mathcal{L}$ is a language and $\mathscr{A}$ is an orderly $\mathcal{L}$-algebra. If $\mathscr{B}$ and $\mathscr{C}$ are reductions of
$\mathscr{A}$ such that $\mathscr{B}(v_i)=\mathscr{C}(v_i) $ for all $i\in \omega$, then $\mathscr{B}=\mathscr{C}$.
\end{corollary}

\begin{proof}
By Theorem~\ref{0310a}, $\mathscr{A}=\mathfrak{A}_{\vec{a}}$ for some $\mathcal{L}$-algebra $\mathfrak{A}$ and sequence
$\vec{a}\in {^\omega}\Vert \mathscr{A}\Vert$. The conclusion then follows by Proposition~\ref{2008a}\textit{(2)}.
\end{proof}

The corollary states that the reduction of any given orderly $\mathcal{L}$-algebra is uniquely determined by its values on the variables.

\section{Ramsey Orderly Algebras}
This section is devoted to the connection between orderly $\mathcal{L}$-algebras and Ramsey algebras. Theorem \ref{1908a} is the main result concerning
the connection.

\begin{definition}
Suppose $\mathcal{L}$ is a language and $\mathscr{A}$ is an orderly $\mathcal{L}$-algebra. We say that $\mathscr{A}$ is \emph{weakly Ramsey} if{f} for
every $X\subseteq \Vert \mathscr{A}\Vert$, there exists a reduction $\mathscr{B}$ of $\mathscr{A}$ \emph{homogeneous for X} in the sense that $\Vert
\mathscr{B}\Vert$ is either contained in or disjoint from $X$. We say that $\mathscr{A}$ is \emph{Ramsey} if{f} for every reduction $\mathscr{B}$ of
$\mathscr{A}$ and $X\subseteq \Vert \mathscr{A}\Vert$, there exists a reduction $\mathscr{C}$ of $\mathscr{B}$ homogeneous for $X$.
\end{definition}


\begin{remark}\label{2110a}
\begin{enumerate}
\item If $\mathscr{A}$ is Ramsey, then it is weakly Ramsey.
\item If some reduction of $\mathscr{A}$ is weakly Ramsey, then $\mathscr{A}$ is weakly Ramsey.
\item If $\mathscr{A}$ is Ramsey, then every reduction of $\mathscr{A}$ is Ramsey.
\end{enumerate}
\end{remark}

\begin{proposition}\label{0710c}
Suppose that $\mathfrak{A}=(A,\mathcal{F})$ is an $\mathcal{L}$-algebra and $\vec{a}\in {^\omega}\!A$.
Then $\mathfrak{A}_{\vec{a}}$ is a Ramsey orderly $\mathcal{L}$-algebra if and only if $\mathfrak{A}$ is Ramsey below $\vec{a}$.
\end{proposition}

\begin{proof}
Proposition~\ref{2008a} will be applied repeatedly.
Assume $\mathfrak{A}_{\vec{a}}$ is Ramsey. Suppose $\vec{b}\leq_{\mathcal{F}}\vec{a}$ and $X\subseteq A$. Then $\mathfrak{A}_{\vec{b}}$ is a reduction
of $\mathfrak{A}_{\vec{a}}$. Choose  a reduction $\mathscr{C}$ of $\mathfrak{A}_{\vec{b}}$ homogeneous for $X\cap \Vert \mathfrak{A}_{\vec{a}}\Vert$.
 We must have $\mathscr{C}= \mathfrak{A}_{\vec{c}}$ for some $\vec{c}\leq_{\mathcal{F}}\vec{b}$.
 By Remark~\ref{2009a}, $\Vert \mathscr{C}\Vert \subseteq \Vert \mathfrak{A}_{\vec{a}}\Vert$. Together with $\Vert
 \mathscr{C}\Vert=\FR_{\mathcal{F}}(\vec{c})$,
 it follows that $\FR_{\mathcal{F}}(\vec{c})$ is
 either contained in or disjoint from $X$. Therefore, $\mathfrak{A}$ is Ramsey below $\vec{a}$.

Conversely, assume $\mathfrak{A}$ is Ramsey below $\vec{a}$.
Suppose  $X\subseteq \Vert\mathfrak{A}_{\vec{a}}\Vert$ and   $\mathscr{B}$ is a reduction of $\mathfrak{A}_{\vec{a}}$, say
$\mathscr{B}=\mathfrak{A}_{\vec{b}}$ for some $\vec{b}\leq_{\mathcal{F}}\vec{a}$.  Choose $\vec{c}\leq_{\mathcal{F}}\vec{b}$ such that
$\FR_{\mathcal{F}}(\vec{c})$ is
either contained in or disjoint from $X$. Then $\mathfrak{A}_{\vec{c}}$ is a reduction of $\mathscr{B}$ homogeneous for $X$ because $\Vert
\mathfrak{A}_{\vec{c}}\Vert=\FR_{\mathcal{F}}(\vec{c})$.
\end{proof}

\begin{theorem}\label{1908a}
Suppose that $\mathfrak{A}=(A,\mathcal{F})$ is an $\mathcal{L}$-algebra. The following are equivalent.
\begin{enumerate}
\item $\mathfrak{A}$ is a Ramsey algebra.
\item $\mathfrak{A}_{\vec{a}}$ is a Ramsey orderly $\mathcal{L}$-algebra for all $\vec{a}\in {^\omega\!}A$.
\item $\mathfrak{A}_{\vec{a}}$ is a weakly Ramsey orderly $\mathcal{L}$-algebra for all $\vec{a}\in {^\omega\!}A$.
\end{enumerate}
\end{theorem}

\begin{proof}
The equivalence of \textit{(1)} and \textit{(2)} follows from Remark~\ref{0102a} and Theorem~\ref{0710c} while
\textit{(2)} immediately implies \textit{(3)}.
Assume \textit{(3)} holds. Fix $\vec{a}\in {^\omega\!}A$. Suppose $X\subseteq \Vert\mathfrak{A}_{\vec{a}}\Vert$ and $\mathscr{B}$ is a reduction of
$\mathfrak{A}_{\vec{a}}$. By Proposition~\ref{2008a}\textit{(2)} and our assumption, $\mathscr{B}$ is weakly Ramsey. Choose a  reduction $\mathscr{C}$
of $\mathscr{B}$ homogeneous for $X\cap \Vert \mathscr{B}\Vert$. This $\mathscr{C}$ is also homogeneous for $X$ as required.
\end{proof}

For the rest of this section, we present an assorted array of elementary results. We begin with showing that if an orderly $\mathcal{L}$-algebra is a
one-to-one function, then it cannot be Ramsey.

\begin{theorem}\label{1709a}
Suppose that $\mathcal{L}$ is a nonempty language and $\mathscr{A}$ is an orderly \mbox{$\mathcal{L}$-algebra}. If $\mathscr{A}$ is one-to-one, then it
is not weakly Ramsey and thus not Ramsey.
\end{theorem}

\begin{proof}
We define a subset $X$ of $\Vert \mathscr{A}\Vert$ as follows.
Fix a function symbol $f\in \mathcal{L}$ and say $f$ is $n$-ary. For every $t\in \OT(\mathcal{L})$, let $\mathscr{A}(t)\in X$ if{f} the number of $f$
occurring before the first variable in $t$ is even.  Since $\mathscr{A}$ is one-to-one, the set $X$ is well-defined.
Suppose $\mathscr{B}$ is any reduction of $\mathscr{A}$, say witnessed by $\langle t_0,t_1,t_2,\dotsc\rangle$.
By definition, $\mathscr{B}(v_0)=\mathscr{A}(t_0)$ and $\mathscr{B}(fv_0v_1\dotsm v_n)=\mathscr{A}(ft_0t_1\dotsm t_n)$.
Obviously, $\mathscr{A}(t_0)\in X$ if and only if $\mathscr{A}(ft_0t_1\dotsm t_n)\notin X$.
Hence, $\mathscr{B}$ is not homogeneous for $X$. Therefore, $\mathscr{A}$ is not weakly Ramsey.
\end{proof}

\begin{corollary}\label{1111a}
If $f$ is a one-to-one binary operation on an infinite set $A$, then $(A,f)$ is not a Ramsey algebra.
\end{corollary}
\begin{proof}[Sketch of proof.]
Through careful analysis, a sequence $\vec{a}$ such that the induced orderly algebra $\mathfrak{A}_{\vec{a}}$ is one-to-one can be constructed
inductively. Hence, by
Theorems \ref{1908a} and \ref{1709a}, $(A, f)$ is not a Ramsey algebra.
\end{proof}

\begin{theorem}\label{0808}
Suppose that $\mathcal{L}$ is a language and $\mathscr{A}$ is an orderly $\mathcal{L}$-algebra with a finite universe. Then $\mathscr{A}$ is weakly
Ramsey if and only if
$\mathscr{A}$ has a reduction that is a trivial orderly $\mathcal{L}$-algebra.
\end{theorem}

\begin{proof}
The backward direction is immediate. Now, suppose $\mathscr{A}$ is weakly Ramsey. Assume $\mathscr{A}$ has no reduction that is trivial.
For this proof, let us say that a reduction $\mathscr{B}$ of $\mathscr{A}$ is minimal if{f} no reduction of $\mathscr{B}$ has a universe with smaller
cardinality.

\begin{claim}
Suppose $\mathscr{B}$ is a minimal reduction of $\mathscr{A}$ and $b\in \Vert \mathscr{B}\Vert$. Then there exists a reduction $\mathscr{C}$ of
$\mathscr{B}$ such that
$\mathscr{C}(v_i)=b$ for every $i\in \omega$.
\end{claim}

\begin{proof}[Proof of claim]
Such $\mathscr{C}$ exists if there are orderly terms $t_0<t_1<t_2<\dotsb$ such that $\mathscr{B}(t_i)=b$ for all $i \in \omega$.
It suffices to show that for every natural number $n$, there exists $t\in \OT(\mathcal{L})$ such that $\mathscr{B}(t)=b$ and the index of the first
variable of $t$ is greater than $n$.  Assume this is not the case. Then the first variable of any $t\in \OT(\mathcal{L})$ such that $\mathscr{B}(t)=b$
 is bounded, say by $N$. Let $\mathscr{C}$ be the reduction of $\mathscr{B}$ witnessed by $v_{N+1}<v_{N+2}<v_{N+3}<\dotsb$.
 Clearly, $b\notin \Vert \mathscr{C}\Vert$ and hence $\Vert\mathscr{C}\Vert$ is a proper subset of $\Vert\mathscr{B} \Vert$, contradicting the
 minimality of $\mathscr{B}$.
\end{proof}

\begin{claim}
If $\mathscr{B}$ and $\mathscr{C}$ are minimal reductions of $\mathscr{A}$, then $\Vert\mathscr{B}\Vert$ and $\Vert\mathscr{C}\Vert$ are either equal or
disjoint.
\end{claim}

\begin{proof}[Proof of claim]
We argue by contradiction. Assume $b\in \Vert \mathscr{B}\Vert \cap \Vert \mathscr{C}\Vert$. Then by the previous claim, there is a reduction
$\mathscr{D}$ of $\mathscr{B}$ such that $\mathscr{D}(v_i)=b$ for all $i\in \omega$.
Likewise, there is a reduction $\mathscr{E}$ of $\mathscr{C}$ such that $\mathscr{E}(v_i)=b$ for all $i\in \omega$.
Since $\mathscr{D}$ and $\mathscr{E}$ are both reductions of $\mathscr{A}$, by Corollary~\ref{0310c}, $\mathscr{D}=\mathscr{E}$.
It follows that $\Vert \mathscr{B}\Vert= \Vert \mathscr{D}\Vert=\Vert\mathscr{E}\Vert= \Vert\mathscr{C}\Vert$.
\end{proof}

By the assumption, the universe of any minimal reduction of $\mathscr{A}$ has size at least two.
Let $X$ consist of exactly one representative from the universe of each minimal reduction of $\mathscr{A}$ such that if two minimal reductions have the
same universe, they share the same representative.
By our claim, $\Vert \mathscr{A}\Vert \backslash X$ contains at least some element from the universe of each minimal reduction of $\mathscr{A}$.
Since $\mathscr{A}$ is weakly Ramsey, choose a reduction $\mathscr{B}$ of $\mathscr{A}$ homogeneous for $X$. Such $\mathscr{B}$ can be further required
to be a  minimal reduction of $\mathscr{A}$. However, this contradicts our choice of $X$.
\end{proof}

The following corollary of Theorem~\ref{0808} has a direct proof in \cite{wcT13}.

\begin{corollary}
Every finite Ramsey algebra is a degenerate Ramsey algebra.
\end{corollary}

The next theorem is a reformulation of Corollary~4.9 in \cite{wcT13a} into the context of orderly algebra.
The proof of the nontrivial backward direction is a minor modification of the proof of Theorem~4.8 in \cite{wcT13a} and is thus omitted.

\begin{theorem}\label{1109a}
Suppose $\mathcal{L}$ is a language that contains some binary function symbol and $\mathscr{A}$ is an orderly $\mathcal{L}$-algebra.
Then $\mathscr{A}$ is weakly Ramsey if and only if for every $X\subseteq \Vert\mathscr{A}\Vert$, there exists a reduction $\mathscr{B}$ of $\mathscr{A}$
\emph{pre-homogeneous for $X$}, in the sense that for every $t_1,t_2\in \OT(\mathcal{L})$ such that the same variables appear in both orderly terms, $
\mathscr{B}(t_1)\in X \text{ if and only if } \mathscr{B}(t_2)\in X$.
\end{theorem}

\begin{remark}
The assumption that $\mathcal{L}$ contains some binary function symbol is necessary. Otherwise, let $\mathfrak{A}=(\mathbb{N},+_3)$, where
$+_3(x,y,z)=x+y+z$ for all $x,y,z\in \mathbb{N}$. If $\vec{a}$ is any infinite sequence of odd numbers, for example, then $\mathfrak{A}_{\vec{a}}$ is
not weakly Ramsey although $\mathfrak{A}_{\vec{a}}$ is trivially pre-homogeneous for any subset of $\Vert \mathfrak{A}_{\vec{a}}\Vert$.
\end{remark}

\begin{corollary}\label{0312b}
Every orderly semigroup is Ramsey.
\end{corollary}

Finally, we present the analogue of Theorem~\ref{0910a} for orderly algebras.

\begin{theorem}\label{0910b}
Suppose $\mathcal{L}$ is a  finite language without unary function symbol and $\mathscr{A}$ is an orderly $\mathcal{L}$-algebra.
If $\mathscr{A}$ is Ramsey, then for every reduction $\mathscr{B}$ of $\mathscr{A}$, $n\in \omega$, and
$X\subseteq \Vert \mathscr{A}\Vert^n$, there exists a reduction $\mathscr{C}$ of $\mathscr{B}$ such that
$\Vert \mathscr{C}\Vert^n_<$ is either contained in or disjoint from $X$,
where
$$\Vert \mathscr{C}\Vert^n_<= \{\,(\mathscr{C}(t_1),\dotsc,\mathscr{C}(t_n) ) \mid t_1,\dotsc,t_n \in \OT(\mathcal{L})\text{ with } t_1<\dotsb<t_n   \,
\}.$$
\end{theorem}

\begin{proof}
Suppose $\mathscr{B}$ is a reduction of $\mathscr{A}$, $n\in \omega$, and
$X\subseteq \Vert \mathscr{A}\Vert^n$.
By Theorem~\ref{0310a}, $\mathscr{A}=\mathfrak{A}_{\vec{a}}$ for some $\mathcal{L}$-algebra $\mathfrak{A}=(\Vert \mathscr{A}\Vert, \mathcal{F})$ and
$\vec{a}\in {^\omega}\Vert \mathscr{A}\Vert$. Hence, by Proposition~\ref{2008a}, $\mathscr{B}=\mathfrak{A}_{\vec{b}}$
for some $\vec{b}\leq_{\mathcal{F}}\vec{a}$.
Since $\mathscr{A}$ is Ramsey, by Theorem~\ref{0710c}, $\mathfrak{A}$ is Ramsey below $\vec{a}$.
By Theorem~\ref{0910a}, there exists $\vec{c}\leq_{\mathcal{F}}\vec{b}$ such that $[\FR_{\mathcal{F}}  (\vec{c})]^n_<$ is either contained in or
disjoint from $X$. Then $\mathfrak{A}_{\vec{c}}$ is a reduction of $\mathscr{B}$ and it remains to note that
$\Vert \mathfrak{A}_{\vec{c}}\Vert^n_<= [\FR_{\mathcal{F}}  (\vec{c})]^n_<$.
\end{proof}

\section{A Case Study}
 In addition to
applying the notion of orderly $\mathcal{L}$-algebras to the study of Ramsey algebras, the case study in this section is intended to demonstrate how the
notion in question facilitates the study of Ramsey algebra. 

Throughout this section, suppose $\mathcal{L}=\{f\}$, where $f$ is binary, and  fix an orderly $\mathcal{L}$-algebra $\mathscr{A}$.
We will define an orderly $\mathcal{L}$-algebra, denoted $\sharp\mathscr{A}$, with universe a subset of $\Vert \mathscr{A}\Vert \times \Vert
\mathscr{A}\Vert$ and show that $\sharp\mathscr{A}$ is Ramsey provided that $\mathscr{A}$ is Ramsey. To do this, every orderly term $t$ of $\mathcal{L}$
is associated with a pair of (orderly) terms of $\mathcal{L}$, denoted $(t^x,t^y)$, defined inductively in the following way:
$$(v_i^x,v_i^y) =(v_{2i}, v_{2i+1}) \text{ for all } i\in \omega;$$
$$( (fst)^x, (fst)^y)= (fs^xs^y, ft^xt^y) \text{ for all } s,t\in \OT(\mathcal{L}) \text{ with } s<t.$$

\begin{claim}
$t^x,t^y\in \OT(\mathcal{L})$ and $t^x<t^y$ for every $t\in \OT(\mathcal{L})$.
\end{claim}

\begin{proof}
In fact, $v_i$ occurs in $t$ if and only if each
of $v_{2i}$ and $v_{2i+1}$ occurs in either $t^x$ or $t^y$.
It is straightforward to prove this stronger claim by induction on the complexity of orderly terms.
\end{proof}

Now, for every $t\in \OT(\mathcal{L})$, define $\sharp\mathscr{A}(t)=(\mathscr{A}(t^x), \mathscr{A}(t^y))$.

\begin{claim}
$\sharp\mathscr{A}$ is an orderly $\mathcal{L}$-algebra and $\Vert \sharp \mathscr{A}\Vert
\subseteq \Vert \mathscr{A}\Vert^2_<$.
\end{claim}

\begin{proof}
Suppose $s_1<t_1$, $s_2<t_2$, $\sharp\mathscr{A}(s_1)=\sharp\mathscr{A}(s_2)$, and $\sharp\mathscr{A}(t_1)=\sharp\mathscr{A}(t_2)$. We need  to show
that $\sharp\mathscr{A}(fs_1t_1)=
\sharp \mathscr{A}(fs_2t_2)$.
From  $\sharp\mathscr{A}(s_1)= \sharp  \mathscr{A}(s_2)$, it follows that
$\mathscr{A}(s_1^x)=\mathscr{A}(s_2^x)$ and $\mathscr{A}(s_1^y)= \mathscr{A}(s_2^y)$. Since $\mathscr{A}$ is an orderly \mbox{$\mathscr{L}$-algebra},
$\mathscr{A}(fs_1^xs_1^y)=\mathscr{A}(fs_2^xs_2^y)$. Similarly, $\mathscr{A}(ft_1^xt_1^y)=\mathscr{A}(ft_2^xt_2^y)$.
Therefore,  $\sharp\mathscr{A}(fs_1t_1)= (\mathscr{A}(fs_1^xs_1^y), \mathscr{A}(ft_1^xt_1^y))=    (\mathscr{A}(fs_2^xs_2^y), \mathscr{A}(ft_2^xt_2^y))=
\sharp \mathscr{A}(fs_2t_2)$ as required.
The second part is immediate.
\end{proof}

\begin{claim}\label{1010a}
Suppose $\mathscr{B}$ is a reduction of $\sharp \mathscr{A}$ witnessed by $\vec{t}=\langle  t_0,t_1,t_2,\dotsc\rangle$.
Let $\mathscr{C}$ be the reduction of $\mathscr{A}$ witnessed by $\vec{t'}=\langle ft_0^xt_0^y,ft_1^xt_1^y,ft_2^xt_2^y,\dotsc\rangle$.
 If $\mathscr{D}$ is a reduction of $\mathscr{C}$,\footnote{This requirement on $\mathscr{D}$ is essential. For example, say $\mathscr{B}=\sharp
 \mathscr{A}$. If
 $\mathscr{D}$ is a reduction of $\mathscr{A}$ witnessed by $v_0<v_2<v_4<\dotsb$, then $\sharp \mathscr{D}$ need not be a reduction of $\mathscr{B}$.}
then $\sharp \mathscr{D}$ is a reduction of $\mathscr{B}$.
\end{claim}

\begin{proof}
First of all, the following substitution property can be proved by induction on the complexity of orderly terms:
$$s[\vec{t'}]=f(s[\vec{t}])^x (s[\vec{t}])^y \text{ for all } s\in \OT(\mathcal{L}).$$
For the base step, $v_i[\vec{t'}]= ft_i^xt_i^y =f(v_i[\vec{t}])^x (v_i[\vec{t}])^y $.
For the induction step,
\begin{multline*}
(fss')[\vec{t'}]=fs[\vec{t'}]s'[\vec{t'}]
= f  f(s[\vec{t}])^x (s[\vec{t}])^y f(s'[\vec{t}])^x (s'[\vec{t}])^y=\\
 f \big( fs[\vec{t}]s'[\vec{t}]   \big)^x\big( fs[\vec{t}]s'[\vec{t}] \big)^y
 = f \big(  (fss')[\vec{t}]   \big)^x\big( (fss')[\vec{t}] \big)^y.
\end{multline*}

Suppose $\mathscr{D}$ is a reduction of $\mathscr{C}$ witnessed by $\vec{u}=\langle u_0,u_1,u_2,\dotsc\rangle$.
Let $\vec{u'}=\langle fu_0u_1,fu_2u_3,fu_4u_5,\dotsc\rangle$.
Again, it can be proved similarly by induction on the complexity of orderly terms that
$$ s[\vec{u'}]= (fs^xs^y)[\vec{u}] \text{ for all } s\in \OT(\mathcal{L}). $$

Now, we will show that $\sharp \mathscr{D}$ is a reduction of $\mathscr{B}$ witnessed by $\vec{u'}$.
Fix $s\in \OT(\mathcal{L})$.
Then by definition, $\sharp \mathscr{D} (s)= ( \mathscr{D}(s^x), \mathscr{D}(s^y))
=( \mathscr{C}(s^x[\vec{u}]), \mathscr{C}(s^y[\vec{u}] ))
= (\mathscr{A} (s^x[\vec{u}] [\vec{t'}  ] ), \mathscr{A} (s^y[\vec{u}] [\vec{t'}  ] )     )$.
By the first substitution property,
$s^x[\vec{u}] [\vec{t'}  ]=f(s^x[\vec{u}][\vec{t}])^x (s^x[\vec{u}][\vec{t}])^y = \big(f (s^x[\vec{u}][\vec{t}])(s^y[\vec{u}][\vec{t}])
\big)^x =   \big(  (fs^xs^y)[\vec{u}][\vec{t}]  \big)^x         $.
Similarly, $s^y[\vec{u}] [\vec{t'}  ]= \big(  (fs^xs^y)[\vec{u}][\vec{t}]  \big)^y $.
Therefore,
$\sharp \mathscr{D} (s) = \sharp \mathscr{A}\big(  (fs^xs^y)[\vec{u}]   [\vec{t}]                       \big)=
\mathscr{B}\big( (fs^xs^y)[\vec{u}] \big) =\mathscr{B}(s[\vec{u'}]    )$ as required.
\end{proof}

\begin{theorem}\label{0510a}
If $\mathscr{A}$ is Ramsey, then $\sharp \mathscr{A}$ is Ramsey.
\end{theorem}

\begin{proof}
Suppose $\mathscr{B}$ is a reduction of $\sharp\mathscr{A}$ witnessed by $\vec{t}=\langle  t_0,t_1,t_2,\dotsc\rangle$ and $X\subseteq \Vert
\sharp\mathscr{A}\Vert$.
Let $\mathscr{C}$ be the reduction of $\mathscr{A}$ witnessed by $\vec{t'}=\langle ft_0^xt_0^y,ft_1^xt_1^y,ft_2^xt_2^y,\dotsc\rangle$.
Since $\mathscr{A}$ is Ramsey and $X\subseteq \Vert \mathscr{A}\Vert^2$, by Theorem~\ref{0910b}, choose a reduction $\mathscr{D}$ of $\mathscr{C}$ such
that $\Vert \mathscr{D}\Vert^2_<$ is either contained in disjoint from $X$. By the previous claim, $\sharp \mathscr{D}$ is a reduction of $\mathscr{B}$.
Finally, since $\Vert \sharp\mathscr{D}\Vert \subseteq \Vert \mathscr{D}\Vert^2_<$, it follows that $\sharp \mathscr{D}$ is homogeneous for $X$.
\end{proof}

\begin{theorem}\label{0107b}
Suppose $g$ is a binary operation on a set $A$ and $h$ is an operation on $A^2$ defined by
$$h((x_1,y_1),  (x_2,y_2)  )=(g(x_1,y_1), g(x_2, y_2)).$$
If $(A,g)$ is a Ramsey algebra, then $(A^2,h)$ is also a Ramsey algebra.
\end{theorem}

\begin{proof}
Let $\mathfrak{A}$ and $\mathfrak{B}$ denote the   $\mathcal{L}$-algebras $(A,g)$ and $(A^2,h)$ respectively.  Suppose $\vec{b}=\langle (x_i,
y_i)\rangle_{i \in \omega}\in {^\omega}\!(A^2)$.
Let $\vec{a}=\langle x_0,y_0, x_1,y_1,x_2,y_2,\dotsc\rangle$.
By Theorems~\ref{1908a} and \ref{0510a}, it suffices to show that
$\mathfrak{B}_{\vec{b}}=\sharp \mathfrak{A}_{\vec{a}}$.
We prove this by induction on the complexity of orderly terms. For the base case, $\mathfrak{B}_{\vec{b}}(v_i)= (x_i,y_i)=
(\mathfrak{A}_{\vec{a}}(v_{2i}), \mathfrak{A}_{\vec{a}}(v_{2i+1}))=\sharp\mathfrak{A}_{\vec{a}}(v_i)$. For the induction step,
consider $fst\in \OT(\mathcal{L})$.
By the induction hypothesis, $s^{\mathfrak{B}}[\vec{b}]=\mathfrak{B}_{\vec{b}}[s]
= \sharp\mathfrak{A}_{\vec{a}}(s)
=\big(\mathfrak{A}_{\vec{a}}(s^x),\mathfrak{A}_{\vec{a}}(s^y)\big)=
\big((s^x)^{\mathfrak{A}}[\vec{a}], (s^y)^{\mathfrak{A}}[\vec{a}]   \big)$.
Similarly, we have $t^{\mathfrak{B}}[\vec{b}]=\big((t^x)^{\mathfrak{A}}[\vec{a}], (t^y)^{\mathfrak{A}}[\vec{a}]   \big)$.
Therefore, \linebreak
$\mathfrak{B}_{\vec{b}}(fst)= f^{\mathfrak{B}}(s^{\mathfrak{B}}[\vec{b}], t^{\mathfrak{B}}[\vec{b}]     )
= h\Big(\, \big((s^x)^{\mathfrak{A}}[\vec{a}], (s^y)^{\mathfrak{A}}[\vec{a}]   \big)\, , \,    \big((t^x)^{\mathfrak{A}}[\vec{a}],
(t^y)^{\mathfrak{A}}[\vec{a}]   \big) \, \Big)
$. By the definition of $h$, this equals $\Big( \,  g\big( (s^x)^{\mathfrak{A}}[\vec{a}], (s^y)^{\mathfrak{A}}[\vec{a}] \big)\, ,\,   g\big(
(s^x)^{\mathfrak{A}}[\vec{a}], (s^y)^{\mathfrak{A}}[\vec{a}]   \big)\, \Big)
= \big( (fs^xs^y)^{\mathfrak{A}}[\vec{a}], (ft^xt^y)^{\mathfrak{A}}[\vec{a}]  \big)
= $
$( \mathfrak{A}_{\vec{a}}( fs^xs^y  )  ,   \mathfrak{A}_{\vec{a}}(ft^xt^y  ) )=\sharp \mathfrak{A}_{\vec{a}} (fst) $.
\end{proof}

\begin{corollary}\label{1021a}
There exists a Ramsey algebra $(A,f)$ where $f$ is a nowhere associative binary operation, meaning $f( f(a,b),c)\neq f(a,f(b,c))$ for every
$(a, b, c)\in A^3$.
\end{corollary}

\begin{proof}
Let $g((x_1,y_1), (x_2,y_2))=(x_1+y_1,x_2+y_2)$ for all $(x_1,y_1), (x_2,y_2)\in \mathbb{N}^2$. It is easy to verify that $g$ is  nowhere
 associative. Since $(\mathbb{N},+)$ is a Ramsey algebra for being a semigroup, by Theorem~\ref{0107b}, $(\mathbb{N}^2, g)$ is a Ramsey algebra.
\end{proof}

The results in this section can be extended analogously to the $n$-ary case.

\section{Concluding Remarks}
The paper begins with a brief historical account of the subject, followed by a formal introduction to the notion of a Ramsey algebra and its connection
with Ramsey spaces.

In an attempt to facilitate the study of the characterization of Ramsey algebras, we have introduced the notion of an orderly $\mathcal{L}$-algebra
based on the observation that whether or not an algebra is a Ramsey algebra can be cast in terms of infinite sequences in the underlying set of the algebra. This observation culminates in the notion of an orderly semigroup, an orderly algebra every semigroup would induce. Such a sufficient condition for an orderly algebra to be Ramsey is contrasted with the fact that even algebras closely resembling semigroups can fail to be Ramsey (cf. \cite{rdz16}), further suggesting that the crucial aspect that determines if an algebra is Ramsey lies in sequences and their reduction
properties.

We ended our paper with a case study to demonstrate the facility afforded by orderly algebras in the study of Ramsey algebra. We conclude the paper with
an invitation to the readers to pursue the study of Ramsey algebras towards a their characterization, particularly to apply the
notion of Ramsey orderly algebras in such a study.





\section*{Acknowledgment}

The notion of Ramsey orderly algebra can be attributed solely to the first author. The authors gratefully acknowledge the support by an FRGS grant
No.~203/PMATHS/6711464 of the Ministry of Education, Malaysia and Universiti Sains Malaysia.


\thebibliography{99}

\bibitem{tC88} T.J.~Carlson, Some unifying principles in Ramsey theory,  Discrete Mathematics~68 (1988), 117--169.
\bibitem{CS84} T.J.~Carlson and S.~G.~Simpson, A dual form of Ramsey's theorem, Adv. in Math.~53 (1984), 265-290.



\bibitem{eE74} E.~Ellentuck, A new proof that analytic sets are Ramsey, J. Symb. Logic~39 (1974), 163-165.







\bibitem{GR71} R.~L.~Graham and B.~L.~Rothschild, Ramsey's theorem for $n$-parameter sets, Trans.~Amer.~Math.~Soc.~{154} (1971), 257-292.

\bibitem{HJ63} A.W.~Hales and R.I.~Jewett, Regularity and positional games, Trans. Amer. Math. Soc.~124 (1966), 360-367.

\bibitem{nH74} N.~Hindman, Finite sums from sequences within cells of a partition of $\mathbb{N}$, J. Combin. Theory (Series A)~17 (1974), 1-11.







\bibitem{HS12} N.~Hindman and D.~Strauss, Algebra in the Stone-\v{C}ech Compactification: Theory and Applications, Second Edition, Walter de Gruyter,
    Berlin (2012).



\bibitem{kM75} K.~Milliken, Ramsey's theorem with sums or unions, J.~Combin.~Theory Ser.~A {18} (1975), 276-290.

\bibitem{rdz16} A. Rajah, W.C. Teh, Z.Y. Teoh, Are Ramsey algebras essentially semigroups, preprint.









\bibitem{aT76} A.~Taylor, A canonical partition relation for finite subsets of $\omega$, J.~Combin. Theory Ser.~A~{21} (1976), 137Ð146.



\bibitem{wcT13b} W.C.~Teh, Ramsey algebras and strongly reductible ultrafilters, Bull. Malays. Math. Sci. Soc.~37(4) (2014), 931-938.
\bibitem{wcT13a} W.C.~Teh, Ramsey algebras and formal orderly terms, Notre Dame J. Form. Log. 58(1) (2017), 115--125.
\bibitem{wcT13} W.C.~Teh, Ramsey algebras, J. Math. Log.16(2) (2016), 16 pages.
\bibitem{wcT13c} W.C.~Teh, Ramsey algebras and the existence of idempotent ultrafilters, Arch. Math. Logic 55 (2016), 475-491.
\bibitem{sT10} S.~Todorcevic, Introduction to Ramsey Spaces, Princeton University Press, 2010.
\end{document}